\documentclass[11pt, a4paper]{article}

\usepackage{amsmath, amsthm, amsfonts, amssymb}

\usepackage{setspace}
\usepackage{fullpage}
\usepackage{enumitem}

\usepackage{tabularx}

\usepackage{floatpag}
\usepackage[dvipsnames]{xcolor}

\usepackage[initials]{amsrefs} 

\usepackage{multirow} % merge rows in a table

\usepackage{float}
\usepackage{caption}
\usepackage{subcaption} % for subfigures

\usepackage{tikz} % tikz graphics
\usetikzlibrary{calc} %calculating distances etc.
\usetikzlibrary{arrows, patterns,snakes} %do nice things in tikz, like underbraces

\usepackage{ifthen}
\usepackage{hyperref}
\hypersetup{colorlinks=true,
	citecolor=blue,
	filecolor=blue,
	linkcolor=blue,
	urlcolor=blue
}
\usepackage{cleveref}

\usepackage{comment}

\newtheorem{lemma}{Lemma}
\newtheorem{theorem}{Theorem} 
\newtheorem{corollary}{Corollary} 
\newtheorem{definition}{Definition}

\newtheorem{proposition}{Proposition}

\newcommand{\floor}[1]{\left\lfloor{#1}\right\rfloor}
\newcommand{\ceil}[1]{\left\lceil{#1}\right\rceil}
\renewcommand{\quote}[1]{\textquotedblleft #1\textquotedblright}

\begin{document}
	%\section*{Unavoidable pairs in hypergraphs}	
%		
%		Erd\H{o}s, Füredi, Rothschild and S\'os initiated  a study of a class of graphs that do not forbid a specific induced subgraph, but rather forbid any induced subgraph on a given number $m$ of vertices and number $f$ of edges. Following their notation we say an $r$-uniform ($r \ge 2$) hypergraph $G$ \emph{arrows} a pair of non-negative integers $(m,f)$ and write $G \to_r (m,f)$ if $G$ has an induced subhypergraph on $m$ vertices and $f$ hyperedges.  We say that  a pair $(n,e)$  of non-negative integers \emph{arrows} the pair $(m,f)$,  and write $(n,e) \to_r (m,f)$, if  for any $r$-graph $G$ on $n$ vertices and $e$ hyperedges, $G \to_r (m,f)$. 
%		A pair $(m,f)$ is called \emph{absolutely r-avoidable} if there is $n_0$ such that for each $n>n_0$ and for any $e\in \{0, \ldots, \binom{n}{r}\}$, $(n,e) \not\to_r  (m,f)$. In this talk we will explore  the existence of absolutely $r$-avoidable pairs for $r\ge 3$ and compare the situation to the $2$-uniform case.\\
\title{Absolutely avoidable order-size pairs in hypergraphs}
\date{\vspace{-5ex}}
\author{
	Lea Weber
	\thanks{
		Karlsruhe Institute of Technology, Karlsruhe, Germany;
		email: \mbox{\texttt{lea.weber@kit.edu}}. The research was partially supported by the DFG grant FKZ AX 93/2-1.
	}
}

\maketitle
%\linespread{1.5}

\begin{abstract}
	\setlength{\parskip}{\medskipamount}
	\setlength{\parindent}{0pt}
	\noindent
	
		For a fixed integer $r\ge 2$, we call a pair $(m,f)$ of integers, $m\geq 1$, $0\leq f \leq \binom{m}{r}$, \emph{absolutely avoidable} if there is $n_0$, such that for any pair of integers $(n,e)$ with $n>n_0$ and  $0\leq e\leq \binom{n}{r}$ there is an  $r$-uniform hypergraph on $n$ vertices and $e$ edges that contains no induced sub-hypergraph on $m$ vertices and $f$ edges. Some pairs are clearly not absolutely avoidable, for example $(m,0)$ is not absolutely avoidable since any sufficiently sparse hypergraph on at least $m$ vertices contains independent sets on $m$ vertices. Here we show that for any $r\ge 3$ and $m \ge m_0$, 
		either the pair $(m, \floor{\binom mr/2})$ or the pair $(m, \floor{\binom{m}{r}/2}-m-1)$ is absolutely avoidable. \\
		Next, following the definition of Erd\H{o}s, Füredi, Rothschild and S\'os, we define the \emph{density} of a pair $(m,f)$ as $\sigma_r(m,f) = \limsup_{n \to \infty} \frac{|\{e : (n,e) \to (m,f)\}|}{\binom mr}$. 
		We show that for $ r\ge 3$ most pairs $(m,f)$ satisfy $\sigma_r(m,f)=0$, and that for $m > r$, there exists no pair $(m,f)$ of density 1. 
\end{abstract}

	\section{Introduction}	
	One of the central topics of graph theory deals with properties of classes of graphs that contain no subgraph isomorphic to some given fixed graph,  see for example Bollob\'as \cite{B}.
	Similarly, graphs with forbidden induced subgraphs have been investigated from several different angles -- enumerative, structural, algorithmic, and more.  \\
	
Erd\H{o}s, Füredi, Rothschild and S\'os  \cite{EFRS} initiated  a study of a class of graphs that do not forbid a specific induced subgraph, but rather forbid any induced subgraph on a given number $m$ of vertices and number $f$ of edges.  Following their notation  we say  an $r$-uniform hypergraph (also referred to as an $r$-graph) $G$ \emph{arrows} a pair of non-negative integers $(m,f)$ and write $G \to_r (m,f)$ if $G$ has an induced sub-hypergraph on $m$ vertices and $f$ hyperedges.  We say that  a pair $(n,e)$  of non-negative integers \emph{arrows} (or simply induces) the pair $(m,f)$,  and write $$(n,e) \to_r (m,f)$$ if  for any $r$-graph $G$ on $n$ vertices and $e$ hyperedges, $G \to_r (m,f)$. We say a pair $(n,e)$ is \emph{realised} by an $r$-graph $G$ if $G$ has $n$ vertices and $e$ edges. If $r$ is clear from the context, we might omit the index and simply write $(n,e)\to (m,f)$.\\

 As an example for $r=2$, if $t_{m-1}(n)$ denotes the number of edges in the balanced complete $(m-1)$-partite graph on $n$ vertices, then by Tur\'an's theorem \cite{T} we know that any graph on $n$ vertices with more than $t_{m-1}(n)$ edges contains $K_{m}$,  a complete subgraph on $m$ vertices. Equivalently stated, we have $(n,e) \to (m,\binom m2)$  if and only if $e > t_{m-1}(n)$. \\
 
Let $r, m, f$ integers, $m \ge r \ge 3$ and $0 \le f\le \binom mr$. 
Following the notation in \cites{EFRS, HMZ}, we define $$\sigma_r(m,f) = \limsup\limits_{n \to \infty}\frac{|\{e: (n,e) \to (m,f)\}|}{\binom nr}.$$
Erd\H{o}s, Füredi, Rothschild and S\'os   \cite{EFRS} considered $\sigma_2(m,f)$ for different choices of $(m,f)$. One of their main results is the following theorem. 
 \begin{theorem}\cite{EFRS}\label{thm:efrs}
 	If $(m,f) \in \{(2,0), (2,1), (4,3), (5,4), (5,6)\}$, then $\sigma_2(m,f) =1 $; otherwise, $\sigma_2(m,f) \le \frac23$.
 \end{theorem}
 The upper bound $\frac 23$  was subsequently improved by He et al.~\cite{HMZ} to $ \frac12$. On the other hand, they showed that there are infinitely many pairs for which the equality $\sigma(m,f) = \frac12$ holds. \\
 
 In \cite{EFRS}, Erd\H{o}s, Füredi, Rothschild and S\'os  also gave a construction that shows that \quote{most of the} $\sigma_2(m,f)$ are $0$, by showing that for large $n$ almost all pairs $(n,e)$ can be realised as  the  vertex disjoint union of a clique and a high-girth graph, and that for fixed $m$ most pairs $(m,f)$ cannot be realised as the vertex disjoint union of a clique and a forest. Axenovich and the author~\cite{AW} investigated the existence of so-called absolutely avoidable pairs $(m,f)$ for which we not only have $\sigma_2(m,f) = 0$, but the stronger property $\{e : (n,e) \to (m,f)\} = \emptyset$ for large $n$. Here, we extend this notion to hypergraphs:
 
\begin{definition} A pair $(m,f)$ is  \textbf{absolutely r-avoidable} if there is $n_0$ such that for each $n>n_0$ and for any $e\in \{0, \ldots, \binom{n}{r}\}$, $(n,e) \not\to_r  (m,f)$.
\end{definition}

In \cite{AW} we showed that for $r=2$ there are infinitely many absolutely avoidable pairs  and amongst others constructed an infinite family of absolutely avoidable pairs of the form $(m, \binom m2/2)$ and showed that for any sufficiently large $m$, there exists an $f$ such that $(m,f)$ is absolutely avoidable. Here, we extend this result to higher uniformities:

\begin{theorem}\label{thm:main}
	Let $r \ge 3$. Then there exists $m_0$ such that that for any $m \ge m_0$ either $(m, \floor{\binom mr/2})$ or $(m, \floor{\binom mr/2} -m-1)$ is absolutely avoidable. 
\end{theorem}

In \cite{EFRS} it was further claimed that \textquotedblleft almost all pairs\textquotedblright\ have $\sigma_2(m,f) = 0$. Here we prove the following:
\begin{proposition}\label{prop:density_zero}
	For $r, m \in \mathbb N$, $r, m \ge 3$, all but at most $m^{\frac{r}{r-1}}$ of all possible $\binom mr$ pairs $(m,f)$ satisfy $\sigma_r(m,f) = 0$.\\ In particular, 
	%for $m$ sufficiently large, all pairs $(m,f)$ for which there is no $x\in [m]$ such that $f = \binom xr$, satisfy $\sigma_r(m,f) = 0$. 
	if for some $l \in \mathbb N$ we have $\binom{l}{r-1} > 2m$, then for $f > \binom lr$ and $f\neq \binom xr$ for $x \in [m]$ we have  $\sigma_r(m,f) = 0 $.
\end{proposition}
As seen in \Cref{thm:efrs}, for $r=2$ there exist pairs with $\sigma_2(m,f) = 1$. This changes for $r \ge3$, as seen in \Cref{prop:density_not_one} below, for which we need some additional definitions and notation. An $r$-graph $G$ is called $l$-partite if the vertex set can be partitioned into $l$ parts $V(G) = V_1 \cup \cdots V_l$, such that for any edge $e \in E(G)$ and $i \in [l]$ we have $|e \cap V_i| \le 1$.
By $T_r(n, l)$ we denote the complete $l$-partite $r$-graph with $n$ vertices and part sizes $n_1, \ldots, n_l \in \{\floor{\frac nl}, \ceil{\frac nl}\}$. Note that for $l< r$, $T_r(l,n)$ is empty, and for $r=2$ this is just the Tur\'an graph. The number of edges in $T_r(n, l)$ is denoted by $t_r(n, l)$ and for $l \ge r$ we have  
$$ t_r(n,l) = \sum_{S \in \binom{[l]}{r}}\prod_{i \in S}n_i   =\frac{(l)_r}{l^r} \binom nr+ o(n^r) ,$$
where $(l)_r = \prod_{i=0}^{r-1} (l-i) = l(l-1)\cdot(l-r+1)$. Note that for $r \ge 2$, $\frac{(l)_r}{l^r} = \frac{(l-1)_r}{(l-1)^r} (1-\frac1{l})^r\frac{l}{l-r}$, and by Bernoulli's inequality we have $(1-\frac1l)^r  > 1- \frac rl$, i.e. $\frac{(l)_r}{l^r} > \frac{(l-1)_r}{(l-1)^r}$, so $\frac{(l)_r}{l^r}$
 is strictly increasing in $l$,  and $\lim\limits_{l \to \infty}\frac{(l)_r}{l^r} = 1$. 
 Also note that we have $\frac{(r)_r}{r^r} = \frac{r!}{r^r} \le \frac1r$.
 
Let $l_{m,r}$ be the largest $l \in \mathbb N$ for which $t_r(m,l) < \frac12 \binom mr$.  Note that this is well-defined by the previous observation, in particular, $l_{m,r} \ge r$. For example, one can verify that for $r=3$, we have $l_{m,r} = \begin{cases} 3, & 4 \le m \le 11, \\
	4, & 12 \le m \le 72,\\
	5, & m \ge 73.
\end{cases}$

\begin{proposition}\label{prop:density_not_one}
		Let $m > r\ge 3$, $0 \le f \le \binom mr$.  Then  $\sigma_r(m,f) < 1$.\\  In particular, we can give the following upper bounds on $\sigma_r$ for any $l \in \mathbb N$ which satisfies $l \le l_{m,r}$:
		\begin{enumerate}
			\item[(a)] We have $ \sigma_r(m, f) \le 1 - \frac{(l)_r}{l^r}$.
			\item[(b)] If $t_r(m,l) < f  <\binom mr - t_r(m,l) $ for some $l$, then $\sigma_r(m,f) < 1 - 2 \frac{(l)_r}{l^r}$.
		\end{enumerate}
\end{proposition}

Note that He, Ma and Zhao~\cite{HMZ} mentioned in their conclusion without proof, that for pairs $(m,f)$ with $m > r\ge 3$, $0 \le f \le \binom mr$, the bound $\sigma_r(m,f)  \le 1 - \frac{r!}{r^r}$ holds.

For some other results concerning sizes of induced subgraphs (of $2$-graphs), see for example Alon and Kostochka~\cite{AK}, Alon, Balogh, Kostochka, and Samotij~\cite {ABKS},  Alon,  Krivelevich, and \mbox{Sudakov}~\cite{AKS}, Axenovich and Balogh~\cite{AB}, Bukh and Sudakov~\cite{BS}, Kwan and Sudakov~\cites{KS1, KS2} and Narayanan, Sahasrabudhe, and Tomon~\cite{NST}. A similar question on avoidable order-size pairs was considered by Caro, Lauri, and Zarb~\cite{CYZ} for the class of line graphs.\\	

%
%The main idea of the proof in \cite{AW} is, that certain pairs $(m,f)$ are not realised by any 2-graph which is the vertex disjoint union of a clique and a forest (or its complement). 
%We then used the observation from \cite{EFRS}, that for any sufficiently large $n$, and any $e \leq c \binom{n}{2}$, for any $0\leq c < 1$, there is a graph on $n$ vertices and $e$ edges that is the vertex disjoint union of a clique and a graph of girth greater than $m$. In particular, any $m$-vertex induced subgraph of such a graph is a disjoint union of a clique and a forest. Considering the complements, we deduced that $(m, f)$ is absolutely avoidable. 

In \Cref{sec2} of this paper we will build on one of the proof ideas used in \cite{AW} for 2-graphs, and extend these methods to higher uniformities in order to  prove \Cref{thm:main}.
In \Cref{sec3} we will make some observations on the $r$-density $\sigma_r$ and prove \Cref{prop:density_zero} and \Cref{prop:density_not_one}.

\section{Existence	of absolutely avoidable pairs}\label{sec2}

For a positive real number $x$, let $[x]=\{0, 1, \ldots, \lfloor x \rfloor\}$.
We will call an $r$-graph $G$ \emph{m-sparse} if every subset of $m$ vertices in $G$ induces at most $m$ edges. We denote the complete $r$-graph or clique on $n$ vertices by $K_{n}^{(r)}$. We call an $r$-graph with at most $m$ edges an \emph{${\le}m$-edge (\textquotedblleft at most $m$-edge\textquotedblright) $r$-graph}. 
	
In order to prove results in the 2-uniform case, the following fact is used in \cites{AW, EFRS}:\\
Let $m>0$ be given. Then for any $v$ large enough there exists a graph of girth at least $m$ on $v$ vertices with $v^{1+\frac{1}{2m}}$ edges. \\
For a probabilistic proof of this fact see for example Bollob\'as~\cite{B}  and for an explicit construction see Lazebnik et al.~\cite{LUW}.

Since in \cite{AW} we only ever used the fact that in such a graph every subset of $m$ vertices induces a forest and thus, spans at most $m-1$ edges, we will use the following for higher uniformities: 

\begin{lemma}\label{lem:probabilistic}
Let $m>0$, $r\ge 2$ be given. Then for any $n$ large enough there exists an $n$-vertex $r$-graph with $\Omega(n^{r-1+\frac1{m+1}})$ edges which is  $m$-sparse.\end{lemma}
\begin{proof}
	Let $c_{m,f,r}= \left(\frac m{e}\right)^{\frac mf}\frac{r!f}{m^re2^{1/f}}$.
	Consider a random $r$-graph $G \in G_r(n,p)$, for $p < c_{m,r,f}n^{-m/f}$. Then the probability that some $m$-subset contains at least $f$ edges is less than
	\begin{align*} \binom nm \binom {\binom m r}f p^f &\le \left(\frac{ne}{m}\right)^m  \left(\frac{m^rep}{r!f}\right)^f  \\
		&< \frac 12.
	\end{align*}
%for 
%$$ p < \underbrace{\left(\frac m{e}\right)^{\frac mf}\frac{r!f}{m^re2^{1/f}}}_{=:\, c_{m,f,r}}\ n^{-m/f}.$$

Now let $X$ be the number of edges in $G$. Using the computations above and the standard bound $\binom nr \ge \left(\frac nr\right)^r$, we have 
$$\mathbb E[X] =  \mathbb E[|E(G)|] = \binom nr p \ge c_{m,f,r} \frac{n^{r-m/f}}{r^r}.$$
 Using Chernoff's bound for $\text{Bin}(n,p)$ distributed random variables, we obtain that for $\delta \in (0,1)$ the probability that $G$ has fewer than a $(1-\delta)$-fraction of the expected number of edges is 
\begin{align*} \mathbb P(X \le (1-\delta)\mathbb E(X) ) 
														 &\le \exp(-\tfrac{\delta^2}{2}\mathbb E[X]) \\
														 & \le \exp(- c_{m,f,r} \frac{\delta^2n^{r-\frac mf}}{2 r^r}  )\\
														 & < \frac12,\end{align*}
where the last inequality holds for $\frac mf \le r$  and $n$ sufficiently large.
	Thus, there exists an $r$-graph on $n$ vertices with at least $ (1-\delta)\binom nr p $ edges in which each $m$-subset spans at most $f-1$ edges.

	Note that, by choosing $f = m+1$, we obtain the existence of an $r$-graph with $c_{r,f,m}' n^{r-1+\frac{1}{m+1}}$ hyperedges and no $m$-subset which spans more than $m$ hyperedges, i.e. an $m$-sparse graph.
\end{proof}

\begin{lemma}\label{yes-clique-forest}
	Let $p, r \in \mathbb N$, $p,r \ge 2$,  and $c$ be a constant, $0\leq c<1$. Then for $n \in \mathbb N$ sufficiently large and any $e \in [c\binom{n}{r}]$,  there exists  a non-negative integer $k$ and an $r$-graph on $n$ vertices and $e$ edges which is the  vertex disjoint union of a  $K_k^{(r)}$ and a $p$-sparse $r$-graph on $n-k$ vertices. \end{lemma}

\begin{proof}
	Let $p, r >0$ be given and let $n$ be a given sufficiently large integer. Let $e\in  [c\binom{n}{r}]$.
	Let $k$ be the non-negative integer such that $\binom{k}{r}\le e \le \binom{k+1}{r}-1$.  Note that since $e \leq c\binom{n}{r}$, $\binom{k}{r} \leq  c\binom{n}{r}$, and thus, $k\leq \sqrt[r]{c}n +1\leq c'n$, where $c'$ is a constant with $c'<1$.
	We claim that the pair $(n,e)$ could be represented as the vertex disjoint union of a $K_k^{(r)}$  and a $p$-sparse $r$-graph. 
	For that, use \Cref{lem:probabilistic} and consider a $p$-sparse $r$-graph $G'$ on $n-k$ vertices with $|E(G')| \geq (n-k)^{r-1+\frac{1}{p+1}}$.
	Consider $G''$, the vertex disjoint union of $G'$ and $K_k^{(r)}$.  Then $|E(G'')| \geq \binom{k}{r} + (n-k)^{r-1+\frac{1}{p}}  \geq \binom{k+1}{r} \geq e$.
	Here, the second inequality holds since \mbox{$\binom{k+1}r - \binom kr = \binom k{r-1}$},  $(n-k)^{r-1+ \frac{1}{p}} \geq \binom k{r-1}$ for  $k\leq c'n$, and since $n$ is sufficiently large.
	Finally, let $G$ be a subgraph of $G''$ with $e$ edges, obtained from $G''$ by removing some edges of $G'$. Thus, $G$ is the vertex disjoint union of a $K_k^{(r)}$ and a $p$-sparse $r$-graph on $n-k$ vertices.
\end{proof}

\begin{lemma}\label{yes-clique-removed}
	Let $p, r \in \mathbb N$, $p,r \ge 2$,  and $c$ be a constant, $0< c\le1$. Then for $n \in \mathbb N$ sufficiently large and any integer $e$ with $ c\binom nr \le e  \le \binom{n}{r}$,  there exists  a non-negative integer $k \le n$ and an $r$-graph on $k$ vertices and $e$ edges which is the complement of a $p$-sparse $r$-graph. \end{lemma}

\begin{proof}
	Note that adding isolated vertices to the complement of an $m$-sparse graph results in the complement of the vertex disjoint union of a clique and an $m$-sparse graph. 
	Thus, the statement immediately follows from \Cref{yes-clique-forest} by taking complements. 
%	Let $p, r >0$ be given and let $n$ be a given sufficiently large integer. Let $c\binom nr \le e\le \binom{n}{r}$.
%	Let $k$ be the non-negative integer such that $\binom{k-1}{r} +1 \le  e \le \binom{k}{r}$.  Note that since $e \ge c\binom{n}{r}$, $\binom{k}{r} \ge  c\binom{n}{r}$, thus $k\ge \frac{\sqrt[r]{c}}en +1\geq c'n$, where $c'$ is a constant, $c'>0$.
%	We claim that the pair $(k,e)$ could be represented as the complement of a  $p$-sparse $r$-graph on $k$ vertices. 
%	For that, use \Cref{lem:probabilistic} and consider a $p$-sparse $r$-graph $G'$ on $k$ vertices with $|E(G')| \geq k^{r-1+\frac{1}{p+1}}$.
%	Consider $G''$, the complement of $G'$. Then $|E(G'')| \le \binom{k}{r} - k^{r-1+\frac{1}{p}}  \le \binom{k-1}{r} \le e$.
%	Here, the second inequality holds since \mbox{$\binom{k}r - \binom {k-1}r = \binom {k-1}{r-1}$},  and $k^{r-1+ \frac{1}{p}} \ge \binom {k-1}{r-1}$.
%	Finally, let $G$ be a subgraph of $G''$ with exactly $e$ edges, obtained from $G''$ by adding some arbitrary edges. Thus, $G$ is the complement of a $p$-sparse $r$-graph on $k$ vertices.
\end{proof}

\begin{lemma}\label{lem:avoidable}
If for some integers $m, r, f$ with $m \ge r \ge 2$ and  $ 0 \le f \le \binom mr$ neither $(m,f)$ nor $(m, \binom mr -f)$ can be realised as an $r$-graph which is the vertex disjoint union of a complete $r$-graph and an ${\le}m$-edge $r$-graph, then the pair $(m,f)$ is absolutely avoidable.
\end{lemma}

\begin{proof}
	Assume we can realise neither $(m,f)$ nor $(m,\binom mr - f)$ as the vertex disjoint union of a complete $r$-graph and an ${\le}m$-edge $r$-graph.
	
	By the previous lemma, for $n$ sufficiently large and any $e\le \ceil{\binom nr/2}$, there exists an $r$-graph $G$ with $e$ hyperedges which is the vertex disjoint union of a clique and an $r$-graph which is $m$-sparse. In particular, for every $ e\in \{0,1, \ldots, \binom nr\}$, there is an $r$-graph $G$ on $n$ vertices with $e$ edges, such that either $G$ or its complement is the vertex disjoint union of a clique and an $m$-sparse $r$-graph.
	
	If $G$ is the union of a clique and an $m$-sparse $r$-graph, then clearly $G \not\rightarrow_r (m,f)$, since $(m,f)$ cannot be realised as the union of a clique and an ${\le}m$-edge $r$-graph.
	
	If $\bar G$ is the union of a clique and an $m$-sparse $r$-graph, then any induced $r$-graph on $m$ vertices is the complement of the vertex disjoint union of a clique and an ${\le}m$-edge $r$-graph. Since $(m, \binom mr-f)$ cannot be realised as the union of a clique and an ${\le}m$-edge $r$-graph, the pair $(m, f) = (m, \binom mr - (\binom mr - f))$ cannot be realised by a graph whose complement is the union of a clique and an ${\le}m$-edge $r$-graph. Thus, $G \not\to_r (m, f)$.
\end{proof}

In the 2-uniform case we used a slightly stronger statement (i.e. no $m$-subset spans more than $m-1$ edges), to find absolutely avoidable pairs. For $r >2$, it suffices to find pairs $(m,f)$, which cannot be realised as the vertex disjoint union of a clique $K_x^{(r)}$ and an ${\le}m$-edge $r$-graph.

Good candidates for such pairs $(m,f)$ are again, as in the 2-uniform case,  pairs which look roughly like $(m, \binom mr/2 + o(1))$. 

We will use the following Lemmata several times:\\

\begin{lemma}\label{mainclaim} Let $r \ge 2, m, f$ be integers with $m \ge r$, $0 \le f \le \binom mr$. If for some $k \in \mathbb N$, $\binom kr + m < f < \binom{k+1}r$, then the pair $(m,f)$ cannot be realised as an $r$-graph which is the vertex disjoint union of a clique and an ${\le}m$-edge hypergraph. 
\end{lemma}
\begin{proof}
Assume $(m,f)$ is realised as $K_l^{(r)}+H$, where $l \ge 0$ and $H$ is an ${\le}m$-edge $r$-graph. Then from the lower bound on $f$, we have that $l > k$, and from the upper bound on $f$, we see that $l < k+1$. Thus, no such $l$ exists. 
\end{proof}

\begin{lemma}\label{mainclaim2}Let $r \ge 2, m, f$ be integers with $m \ge r$, $0 \le f \le \binom mr$. If for some $k \in \mathbb N$, $\binom {k-1}r  < f < \binom{k}r - m$, then the pair $(m,f)$ cannot be realised as an $r$-graph which is the union of the complement of an ${\le}m$-edge hypergraph and some isolated vertices. 
\end{lemma}
\begin{proof}
	Assume $(m,f)$ is realised as $K_l^{(r)}-H$, where $l \ge 0$ and $H$ is an ${\le}m$-edge $r$-graph, and some isolated vertices. Then from the upper bound on $f$, we have that $l < k$, and from the lower bound on $f$, we see that $l > k-1$. Thus, no such $l$ exists. 
\end{proof}

\begin{proof}[Proof of \Cref{thm:main}]
	Let $r\ge3$, $m \ge m_0$ and let $f_0 = \floor{\binom mr /2}$
	% and assume that $f_0 \in \mathbb N$
	.\\
	Using \Cref{lem:avoidable}, we need to show that either $(m,f_0)$ or both $(m, f_0-(m+1))$ and $(m, f_0 + (m+1))$ are not realisable as the vertex disjoint union of a clique and an ${\le}m$-edge $r$-graph. To this end we will show that the condition of \Cref{mainclaim} is satisfied. \\
	Let $x$ be an integer such that $\binom xr \le \floor{\binom mr /2} < \binom{x+1}r$. By standard bounds we observe the following: 
	%$$\left(\frac xr\right)^r \le \binom xr \le f_0 = \binom mr /2 < \frac12 \left(\frac {me}r\right)^r \quad \Rightarrow \quad x < 2^{-1/r}me,$$
	$$\frac12\left(\frac mr\right)^r \le \floor{\frac12 \binom mr }< \binom {x+1}r  <  \left(\frac {(x+1)e}r\right)^r ,$$
	and thus, 
	$$x+1 > \frac 1{2^{1/r}e} m.$$
	Thus, by choosing $m_0$ sufficiently large, we have for $m \ge m_0$ that $x-1 >  \frac m4$, since for $r \ge 3$ $\sqrt[r]{2}e \ge \sqrt[3]{2}e > 4$, and 
	 \begin{equation}\binom{x-1}{r-1} \ge \left(\frac{x-1}{r-1}\right)^{r-1} \ge \left(\frac{m }{4(r-1)}\right)^{r-1} >   2m+2. \tag{$*$}\end{equation}
%	\textbf{Observation:} 

%	and

	\textbf{Case 1:} $\binom xr + m < f_0$. Then by  \Cref{mainclaim}, $(m,f_0)$ cannot be realised by $K_k^{(r)} + H$, where $k \in \mathbb N$ and $H$ has at most $m$ edges. If $\ceil{\binom mr/2} < \binom{x+1}{r}$, then again by \Cref{mainclaim}, both $(m, f_0)$ and $(m, \binom mr - f_0)$ cannot be realised as the disjoint union of a clique and an ${\le}m$-edge $r$-graph, i.e. by \Cref{lem:avoidable}, $(m,f_0)$ is absolutely avoidable. \\
	Otherwise, we have $\ceil{\binom mr /2} = \binom {x+1}r = \floor{\binom mr/2}+1$.
	Let $f_- = \floor{\binom mr/2} - (m+1) $ and $f_+ = \ceil{\binom mr/2} + (m+1)$. We clearly have $f_- < \binom{x+1}r$ and $f_+ > \binom{x+1}r + m$, so it remains to show that $f_- > \binom{x}r + m$ and $f_+ < \binom{x+2}{r}$. Indeed, we have 
	\begin{align*} f_- - \binom{x}r  &= \binom {x+1}r - m -1 - \binom{x}{r}  \\
		&= \binom{x}{r-1} - (m+1) \\
		& \overset{(*)}{>} 2m+1 - m -1 = m,
	\end{align*} i.e. $f_- > \binom xr + m$, and also 
	$$f_+ = \binom {x+1}r + m+1 < \binom{x+1}{r} + 2m+1 \overset{(*)}{<}  \binom{x+1}{r}+\binom{x-1}{r-1} < \binom{x+2}{r},$$
	and thus, $f_+ < \binom{x+2}{r}$.
 \\
	
	\textbf{Case 2:} $\binom xr \le \floor{\binom mr/2} %\le \ceil{\binom mr/2} 
	\le \binom xr + m$.\\
	Let $f_- = \floor{\binom mr/2} - (m+1) $ and $f_+ = \ceil{\binom mr/2} + (m+1)$.
	%$\binom xr + t = f_0$, $0 \le t\le m$. Let $s= \max\{t, m-t\}+1$. Consider $f_- = f_0-s$ and $f_+ = f_0+s$. {\color{red} be careful with $f_-, f_+$ and parity}\\
	 It remains to check that neither $(m,f_-)$ nor $(m, f_+)$ can be realised as the vertex disjoint union of a clique and an  ${\le}m$-edge $r$-graph. On the one hand we have $f_-  \le \binom xr + m-(m+1) < \binom xr$. Thus, in order to use \Cref{mainclaim}, it remains to verify that we also have $f_- > \binom{x-1}{r} + m$. Indeed, we have 
	\begin{align*} f_- - \binom{x-1}r  &\ge \binom xr - (m+1) - \left(\binom xr - \binom{x-1}{r-1}\right)  \\
		&= \binom{x-1}{r-1}  - (m+1) \\
		& \overset{(*)}{>} 2m+1 - m -1 = m.
	\end{align*}
	Thus, we have $f_- - \binom{x-1}r >m$ for $m \ge m_0$, i.e. we have 
	$\binom{x-1}r + m< f_- < \binom{x}{r}$, so by \Cref{mainclaim}, $(m,f_-)$ cannot be realised as the vertex disjoint union of a clique and an ${\le}m$-edge $r$-graph.
	
	On the other hand, we clearly have  $f_+ > \binom xr + m$, and also, $$f_+ \le \binom xr + m+1 + (m+1)  \overset{(*)}{<} \binom xr + \binom{x-1}{r-1}  < \binom xr + \binom{x}{r-1} =  \binom{x+1}{r},$$
	so by \Cref{mainclaim}, $(m,f_+)$ cannot be realised as $K_k^{(r)}+H$, where $k \in \mathbb N$ and $H$ is an ${\le}m$-edge $r$-graph.
	
	Thus,  by \Cref{lem:avoidable} the pair $(m, f_-)$ is absolutely avoidable. %Note that in particular all inequalities hold, if we instead choose $s = m+1$, which proves the statement.
\end{proof}

\section{Density observations}\label{sec3}
Let $r \ge 3$, $m$, $f\le \binom mr$. Recall from the introduction that the density is defined as 
 $$\sigma_r(m,f) = \limsup\limits_{n \to \infty}\frac{|\{e: (n,e) \to (m,f)\}|}{\binom nr}.$$
 
By considering complements, it immediately follows that  $\sigma_r(m,f) = \sigma_r(m, \binom mr - f)$. 
Recall that for a family of $r$-graphs  $\mathcal G$, $\operatorname{ex}_r(n, \mathcal G)$ denotes the extremal number, i.e. the maximum number of edges an $r$-graph on $n$ vertices can have without
containing any member of $\mathcal G$. For an $r$-graph $H$, the \emph{Tur\'an-density} is defined 
as $\pi_r(H) = \lim\limits_{n \to \infty}\frac{\operatorname{ex}(n, \{H\})}{\binom nr}$.

Note that for $f = 0$,  $\sigma_r$ corresponds to the Tur\'an density, i.e. 
$ \sigma_r(m,0) = \sigma_r(m,\tbinom mr) = \pi_r(K_{m}^{(r)}) $,
 where the currently best known general bounds on the Tur\'an density are 
$$ 1- \left(\frac{r-1}{m-1}\right)^{r-1}\le \pi(K_{m}^{r}) \le 1 - \binom{m-1}{r-1}^{-1},$$
due to Sidorenko~\cite{S81} and de Caen~\cite{dC}. Also note that
$\sigma_r(r, 1) = \sigma_r(r,0) = 1$. Thus, the only non-trivial cases are $m > r$, which are dealt with in \Cref{prop:density_zero} and \Cref{prop:density_not_one}. 

Before we prove \Cref{prop:density_zero}, we show the following auxiliary lemma:

\begin{lemma}\label{lem:not_dense}
	Let $m, r, f \in \mathbb N$ with $m \ge r \ge 3$ and $0 \le f \le \binom mr$. 
	\begin{enumerate}
		\item[(a)] If $(m,f)$ cannot be realised as the disjoint union of a clique and an ${\le}m$-edge $r$-graph, then $\sigma_r(m,f) = 0$. In particular, if there is no $x\in [m]$, sucht that $0 \le f - \binom xr < m$, then $\sigma_r(m,f) = 0$.
		\item[(b)] If $(m,f)$ cannot be realised as the complement of an ${\le}m$-edge $r$-graph and some isolated vertices, then $\sigma_r(m,f) = 0$. In particular, if there is no $x\in [m]$, sucht that $0 \le \binom xr - f < m$, then $\sigma_r(m,f) = 0$.
		\item[(c)] 	If $\sigma_r(m,f) > 0$, then there exist $x, \bar x \in [m]$ such that $0 \le f - \binom{x}{r} < m$ and\\ \mbox{$0 \le (\binom mr - f)- \binom{\bar x}{r} < m$}.
	\end{enumerate}
\end{lemma}

\begin{proof}
	\begin{enumerate}
		\item[(a)] By \Cref{yes-clique-forest}, for any $0 < c' < 1$, $n$ sufficiently large, and $e \in \mathcal E_n := \left[c'\binom nr\right]$ there exists an  $r$-graph $G$ on $n$ vertices with $e$ edges which is the vertex disjoint union of a clique and an $m$-sparse $r$-graph. %Note that we have $\left|\mathcal E_n\right| = \binom{c'n}{r} = \binom nr - O(n^{r-1/(m+1)})$. 
		
		Note that any induced subgraph on $m$ vertices of $G$ is the union of a clique and an $r$-graph with at most $m$ edges. Thus, by  definition of $\sigma_r$, if a pair $(m,f)$ cannot be realised by a clique and an ${\le}m$-edge $r$-graph, we have $$\sigma_r(m,f) = \limsup\limits_{n \to \infty}\frac{|\{e: (n,e) \to (m,f)\}|}{\binom nr} \le \limsup\limits_{n \to \infty} \frac{|\binom{[n]}{r}- \mathcal E_n|}{\binom nr} = \frac{(1-c')\binom{n}{r}}{\binom nr}=  1-c'.$$
		Letting $c'$ go to one proves the statement.
		\item[(b)] By \Cref{yes-clique-removed}, for any $0 < c' < 1$, $n$ sufficiently large, and $e \in \mathcal E_n := \left[\binom  nr \right] - \left[c'\binom nr\right]$ there exists an  $r$-graph $G$ on $n$ vertices with $e$ edges which is the complement of an $m$-sparse $r$-graph and some isolated vertices.
		
		Note that any induced subgraph on $m$ vertices of $G$ is the union of a clique with at most $m$ edges removed and an empty graph. Thus, by  definition of $\sigma_r$, if a pair $(m,f)$ cannot be realised as the complement of an ${\le}m$-edge $r$-graph and some isolated vertices, we have $$\sigma_r(m,f) = \limsup\limits_{n \to \infty}\frac{|\{e: (n,e) \to (m,f)\}|}{\binom nr} \le \limsup\limits_{n \to \infty} \frac{|\binom{[n]}{r}- \mathcal E_n|}{\binom nr} = \frac{c'\binom{n}{r}}{\binom nr}=  c'.$$
		The statement follows by letting $c'$ go to zero.
		
		\item[(c)] Now assume we have $\sigma_r(m,f) > 0$. Let $y$, $\bar y$ be the largest integers with $y \le f$, $\bar y \le \bar f$ which can be written as binomial coefficients $y = \binom xr$, $\bar y = \binom{\bar x}r$ for some $x, \bar x \in [m]$. Then applying part (a)  to $(m,f)$ and $(m,\bar f)$, we have $f - y <m$ and $\bar f -\bar y<m$.
	\end{enumerate}	
\end{proof}

\begin{proof}[Proof of \Cref{prop:density_zero}]

By \Cref{lem:not_dense}(a) and (b), it remains to show that for fixed $m$ at most $m^{\frac{r}{r-1}}$ of all possible $\binom mr$ pairs $(m,f)$ can be realised as both the vertex disjoint union of a clique and an ${\le}m$-edge $r$-graph and the complement of an ${\le}m$-edge $r$ graph and some isolated vertices. Let $m$ be fixed and let $f \le \binom mr$ and write $f$ uniquely as $f = \binom lr + l'$, where $0 \le l' < \binom l{r-1}$.
By \Cref{mainclaim},  if  $l' > m$, then $(m,f)$ cannot be realised as the vertex disjoint union of  $K_l^{(r)}$ and an ${\le}m$-edge $r$-graph. %There also does not exist such a realisation if $0 < m-l < r$, since then we cannot realise $(m,f)$ as an $l$ clique and a graph with $l'$ edges, since we cannot have any $r$-edge using strictly fewer than $r$ vertices. 
 By \Cref{mainclaim2}, if  $l' < \binom {l}{r-1} - m$, i.e. $f < \binom{l+1}{r}-m$,  then $(m,f)$ cannot be realised as the complement of an ${\le}m$-edge $r$-graph and some isolated vertices. 
In particular, if $l' > m$ or $l' < \binom{l}{r-1} -m$ and $l \ne 0$, by \Cref{lem:not_dense} we have $\sigma_r(m,f) = 0$. Thus, if for $l$ we have $ m < \binom{l}{r-1} - m$, then any pair $(m,f)$ with $f \ge \binom{l}{r}$ and $f \ne \binom xr$ for some $x \in [m]$,  satisfies $\sigma_r(m,f) = 0$. 
In particular, this means that we have $\sigma_r(m,f) = 0$ if $l \in \Omega(m^{\frac1{r-1}})$ and $l' > 0$. 
 
 %Thus, in particular, by \Cref{lem:not_dense} for $l' \in [1, r-1] \cup \left[m+1,\binom l{r-1}-1 \right]$, we have $\sigma_r(m,f) = 0$.\\ 
 Now assume we have $l \in o(m^{\frac{1}{r-1}})$, 
 $f=\binom lr + l'$,  $0 \le l < m$, $0 \le l' < \binom l{r-1}$. Then, for each of the possible choices of $l$, at most $m$ choices of $l'$ satisfy $l' \le m$, while for all others we obtain $l' > m$. Thus, again by \Cref{lem:not_dense}(a), these pairs cannot be realised as the vertex disjoint union of a clique and an ${\le}m$-edge $r$-graph.
 
 Thus, since $r\ge 3$, all but at most $m\cdot m^{\frac1{r-1}} + m  \in O(m^{\frac r{r-1}})$  of all possible $\binom mr$ pairs cannot be realised as the vertex disjoint union of a clique and an $\{le\}m$-edge $r$-graph, so at least $\binom mr - m^2$ of  all pairs $(m,f)$ have density $\sigma_r(m,f) = 0$. Note that for $r \ge 3$, we have $m^{\frac{r}{r-1}} \in o(\binom mr)$.
\end{proof}

For the proof of \Cref{prop:density_not_one} we will use the following extension of Tur\'an's theorem to hypergraphs by Mubayi~\cite{M}. For fixed $l, r  \ge 2$ let $\mathcal F_l^{(r)}$ be the family of $r$-graphs with at most $\binom l2$ edges, that contain a \emph{core} $S$ of $l$ vertices, such that every pair of vertices in $S$ is contained in an edge.

\begin{theorem}\label{thm:ex}\cite{M}[Mubayi '05]
	Let $r, l, n \ge 2$. Then 
	$$\operatorname{ex}(n, \mathcal F_{l+1}^{(r)}) =  t_r(n, l) %= \prod_{i=1}^r \floor{\frac{n+i-1}{r}}  \left(1-\frac1{r^{r-1}}\right) \binom nr+ O(1)=\frac{r!}{r^r} \binom nr+ o(n^r)
	$$
	 and the unique $r$-graph on $n$ vertices containing no copy of any member of $\mathcal F_{l+1}^{(r)}$ for which equality holds is $T_r(n,l)$, the complete  balanced $l$-partite $r$-graph on $n$ vertices. 
\end{theorem}

\begin{proof}[Proof of \Cref{prop:density_not_one}]
Now let $m > r \ge 2$, $0  \le f \le \binom mr$.\\
%By the observation above, $\frac{(l)_r}{l^r}$ is increasing. Now by definition, $l_{m,r}$ is the largest $l \in \mathbb N$ for which $\frac{(l)_r}{l^r} \le \frac12$, i.e. $t_r(m,l_{m,r}) \le \frac12 \binom mr < t_r(m, l_{m,r}+1)$. 
Let $l \in \mathbb N$, such that $t_r(m,l) < \frac12 \binom mr$.
Note that for $r \ge 3$, such an $l$ always exists, since we have $t_r(m,r) < \frac12 \binom mr$, so we can always choose $l=r$.\\ 
Thus, in particular, we are in one of two cases: Either we have $f \ge \frac12\binom mr > t_r(m,l)$, or we have $f \le \frac 12\binom mr$, i.e. $f-\binom mr \ge \frac12 \binom mr > t_r(m,l)$. \\
\textbf{Case 1:}  $f > t_r(m,l)$. Then by \Cref{thm:ex}, any $r$-graph that realises the pair $(m,f)$ contains a member of $\mathcal F_{l+1}^{(r)}$. Thus, in order to have $(n,e) \to_r (m,f)$, we must have $e > t_r(n,l)$, since for $e \le t_r(n,l)$ we could simply take a subgraph of $T_r(n,l)$ with $e$ edges which does not contain any member of $\mathcal F_{l+1}^{(r)}$, and hence no graph that realises $(m,f)$. In particular, this implies that $$\sigma_r(m,f) \le \lim\limits_{n \to \infty}\frac{\binom nr - t_r(n,l)}{\binom nr} <1 .$$\\
\textbf{Case 2:}  $\binom mr - f > t_r(m,l)$. Then by \Cref{thm:ex} any graph that realises $(m, \binom mr -f)$ contains a member of $K_{l+1}^{(r)}$. Then any graph $G$ with $G \to_r (m, \binom mr -f)$ must contain a member of $\mathcal F_{l+1}^{(r)}$, i.e. $|E(G)| > t_r(n, l)$. Thus, for each $e\le t_r(n, l)$, $(n, e) \not\to_r (m, \binom mr - f)$, and thus, by considering the complement, for each $e \ge \binom nr - t_r(n,l)$, $(n,e) \not\to_r (m,f)$. In particular, we have 
$$ \sigma_r(m,f) \le \lim_{n \to \infty} \frac{\binom nr - t_r(n,l)}{\binom nr}<1.$$
%In order to have $(n,e) \to_r (m,f)$, we must have $\binom nr-e > t_r(n,r)$. Thus, in this case we again obtain that $\sigma_r(m,f) < 1$.
\\
Thus, in either case we have  
 $$\sigma_r(m, f) \le 1- \limsup\limits_{n \to \infty}\frac{t_r(n,l)}{\binom nr} = 1- \frac{(l)_r}{l^r} < 1.$$
 This proves part (a).
 
 To obtain part (b), assume that for some $l \in \mathbb N$ we have $t_r(m,l) < f < \binom mr - t_r(m,l)$. Then by Cases 1 and 2, we see that  $(n,e) \to_r (m,f)$ requires  $t_r(n,l) < e <  \binom nr - t_r(n,l)$. Thus, we obtain that 
$$\sigma_r(m,f) \le 1 - 2\limsup_{n \to \infty}\frac{t_r(n,l)}{\binom nr} = 1 - 2\frac{(l)_r}{l^r},$$
which completes the proof. 
\end{proof}
 
 \begin{corollary}\label{cor:small_r}
 	For $r=3$, $m> 3$, $0<f< \binom m3$, we have the following upper bounds on $\sigma_3(m,f)$:
 	\begin{enumerate}
 		\item $\sigma_3(m,f) \le \frac79$.
 		\item If $t_3(m,3) < f < \binom mr - t_3(m,3)$, then $\sigma_3(m,f) \le \frac59$.
 		\item If $m \ge 12$, then $\sigma_3(m,f) \le \frac58$.
 		\item If $m\ge 73$, then $\sigma_3(m,f) \le \frac{13}{25}$.
 	\end{enumerate}
 For $r=4$, $m > 4$, $0 < f < \binom m4$, we have the following upper bounds on $\sigma_4$:
 	\begin{enumerate}
 		\item $\sigma_4(m,f) \le \frac{29}{32}$, 
 		\item There is $m_0$, such that for all $m \ge m_0$ we have $\sigma_4(m,f) \le \frac{131}{243} \approx 0.54$. 
 	\end{enumerate}
 \end{corollary}
\begin{proof}
	We start with $r=3$. In order to obtain our bounds, we can compute the fraction $\frac{(l)_3}{l^3}$ for different $l \ge 3$. We have that  
	$$\frac{(3)_3}{3^3} = \frac29, \qquad \frac{(4)_3}{4^3} = \frac{3}{8}, \qquad \frac{(5)_3}{5^3} = \frac{12}{25}, \qquad \frac{(6)_3}{6^3} = \frac59.$$
	Note that $\frac{(6)_3}{6^3} > \frac12$, so for $r=3$, the best possible upper bound one can achieve for any pair using \Cref{prop:density_not_one} will use $l=5$. \\
	Now $(1)$ and $(2)$  immediately follow from \Cref{prop:density_not_one}, by setting $l = r = 3$ and observing that for $r\ge 3$, we always have $t_r(m,r) < \frac12 \binom mr$. Then 
$$\sigma_3(m,f) \le  1 - \lim_{n \to \infty} \frac{t_3(n, 3)}{\binom n3} = 1 - \lim_{n \to \infty }\left(\frac{(l)_r}{l^r} \frac{\binom nr}{\binom nr}+ \frac{o(n^r)}{\binom nr}\right)  = 1 - \frac{(3)_3}{3^3} = \frac79,$$
and for $(2)$, if $t_3(m,3) < f < \binom mr - t_3(m,3)$, 
$$\sigma_3(m,f) \le 1 - 2 \lim_{n \to \infty} \frac{t_3(n, 3)}{\binom n3} = \frac59.$$
Now for $(3)$ note that for $m \ge 12$, we have $l_{m,3} = 4$, and thus, by \Cref{prop:density_not_one}(a), for all pairs $(m,f)$ with $m\ge 12$ we have 
$$\sigma_3(m,f) \le 1 - \left(\frac{(4)_3}{4^3}\right) = 1-\frac{6}{16} =\frac58.$$
For $(4)$ note that for $m \ge 73$, we have $l_{m,3}=5 $,  so using \Cref{prop:density_not_one}(a), for all pairs $(m,f)$ with $m\ge 73$ we have $$\sigma_3(m,f) \le 1 - \left(\frac{(5)_3}{5^3}\right) = 1-\frac{12}{25}= \frac{13}{25}.$$

For the case $r=4$, we obtain the first part by computing $\frac{(4)_4}{4^4} = \frac{3}{32}$. The second part is obtained by computing $\frac{(9)_4}{9^4} = \frac{112}{243}$ and noting that $l_{m,4} = 9$ for $m \ge m_0$.

\end{proof}

\section{Conclusion}

We have proven that for $m > r \ge 3$ and $f$ with $0 \le f \le \binom mr$, we always have $\sigma_r(m,f) < 1$. We have also shown that for fixed $r$, most pairs $(m,f)$ satisfy $\sigma_r(m,f) = 0$. On the other hand, for $r \ge 3$ there is no pair $(m,f)$  with $0 < f < \binom mr$ for which we can show that $\sigma_r(m,f) > 0$. This inspires the following question:\\

 \textbf{Open problem:} For $m > r \ge 3$, are there any $f$ with $0 <  f < \binom mr$ such that $\sigma_r(m,f) > 0$?\\

% 
% $f= \binom lr + l'$ for some $l \in [m], 
% 
% $s, l \in [m]$ which satisfy 
% $$ \binom mr =   \binom lr +\binom sr +  s'+l', \quad s', l' \in \{0\} \cup [r,m].$$

% 
%Assume such a pair $(m,f)$ with $\sigma_r(m,f) > 0$ exists. Write $f= \binom lr +l'$, $0 \le l' < \binom{l}{r-1}$  and note that by applying \Cref{mainclaim} and \Cref{lem:not_dense}, we must have  $r \le l' \le m$.
%
%Consider the complementary pair $(m, \bar f$) with $\bar f = \binom mr - f = \binom mr - \binom lr - l' = \binom sr + s'$, where $0 \le s' \le \binom s{r-1}$. If $s'>m$, then again by \Cref{mainclaim} and \Cref{lem:not_dense}, we have $\sigma_r(m,f) = 0$. Thus, the observation follows. 

%\Cref{lem:not_dense}(a) for example shows that we have $\sigma_r(m,f) = 0$ if $f < \binom{m-1}{r-1} - m$, since then $\bar f =\binom mr - f > \binom{m-1}{r} + m$, i.e. $(m,f)$ cannot be realised as the disjoint unoion of a clique and an ${\le}m$-edge $r$-graph.\\

Next, we will use \Cref{lem:not_dense} to identify candidate pairs $(m,f)$ for $r=3, m \le 15$ which might satisfy $\sigma_3(m,f) >0$:\\ 

%\subsubsection*{Possible pairs of positive density for $r=3$}

\begin{lemma}\label{lem:classification}
Let $r=3$, $4 \le m \le 15$ and $0 < f < \binom mr$. If $(m,f) \neq (6,10)$, then $\sigma_3(m,f) = 0 $.
\end{lemma}

\begin{proof}
Let $(m,f)$ be a pair with $4 \le m \le 15$, $0 < f < \binom mr$ with $\sigma_3(m,f) > 0$. 
One can verify that for $4 \le  m \le 15$, we have $2\binom{m-3}3 +4\le \binom m3$, i.e. $\binom {m-3}{3}+2 \le \binom m3 - \binom{m-3}3 -2$. Thus, for any $f \le \binom m3$, we either have $f \ge \binom{m-3}{3} +2$ or $\binom m3 - f \ge \binom {m-3}3 + 2$, assume w.l.o.g. that we have $f \ge \binom{m-3}{3} +2$.
Let $x$ be the unique value with $\binom x3 \le f < \binom{x+1}3$ and write $f = \binom x3 + x'$, $0 \le x' < \binom{x}{2}$. Then by assumption we have $x \in \{m-3, m-2, m-1\}$. 

By \Cref{lem:not_dense}(a) we know that if we cannot realise the pair $(m,f)$ as the vertex disjoint union of a clique and an ${\le}m$-edge $3$-graph, then we have $\sigma_3(m,f) = 0$. Thus, we can assume that the pair $(m,f)$ can be realised as the vertex disjoint union of a clique and an ${\le}m$-edge $3$-graph. By \Cref{mainclaim}, it follows that $(m,f)$ can be realised as the disjoint union of $K_x^{(3)}$ and an ${\le}m$-edge $3$-graph.

Assume we have $x=m-3$. Since by assumption $f \ge  \binom {x}3 +2$, clearly the pair $(m,f)$ cannot be realised as the vertex disjoint union of $K_x^{(3)}$ and a graph on at most $1$ edge. Since $m-x = 3$, it also cannot be realised as the vertex disjoint union of $K_{x}^{(3)}$ and a graph with $2$ edges, a contradiction. Thus,  for $x=m-3$ we have $\sigma_3(m,f) = 0$.

So assume $x\in \{m-2, m-1\}$. Note that since $m- x \le 2$, the pair $(m,f)$ cannot be realised as the vertex disjoint union of a clique on $x$ vertices and a graph with at least $1$ edge. Thus, the only pairs $(m,f)$ which might satisfy $\sigma_3(m,f) > 0$ have $x'=0$, i.e. $f \in \{\binom{m-2}3, \binom{m-1}3\}$. Then $(m,f) \in \mathcal A_1 \cup \mathcal A_2 \cup \{(6,10), (10, 84), (13, 165), (15,169), (15,91)\}$ with 
\begin{align*}\mathcal A_1 &= \{(5,4), (7,20), (8,35), (9,56), (11, 120),  (12, 165), (13,220), (14, 286)\}\\
	\mathcal A_2 &= \{(5,1), (6,4), (7,10), (8,20), (9,35), (10,56), (11,84),  (12, 120), (14, 220) \}.\end{align*}
Now let $(m,f) \in  \mathcal A_1 \cup \mathcal A_2$. let $\bar f = \binom m3 - f$ and let $y \in [m]$ such that $\binom y3 \le \bar f < \binom{y+1}3$, i.e. $\bar f = \binom y3 + y'$ for some $y' \le \binom{y}{2}$. 
 Then it is easy to verify that we are in one of three cases: 
 \begin{itemize}
 	\item $y\in \{m-1, m-2\}$ and $y' > 0$, 
 	\item $y = m-3$ and $y' >1$,
 	\item $y \le m-4$ and $y' > m$. 
 \end{itemize}
In each case, by \Cref{mainclaim} $(m,\bar f)$ canot be realised as the disjoint union of a clique and an ${\le}m$-edge $3$-graph, and thus, by \Cref{lem:not_dense}(a), $\sigma_3(m,f) = \sigma_3(m, \bar f) = 0$. 

The pair $(6,10)$ is self-complementary with $10 = \binom53 = \binom 63/2$. 

For the pair $(10, 84)$ we have $\binom{10}3 - 84 = 36 = \binom 73 +1 = \binom {10}3 - \binom 93$. Note that for $m = 10$, we have $2m < \binom72$, i.e. by \Cref{prop:density_zero}, $\sigma_3(10,36) = 0$.

For the pair $(13, 165)$ we have $\binom{13}3 -165 = 121= \binom{10}3 +1 = \binom {13}3 - \binom{11}{3}.$ Note that for $m=13$, we have $2m < \binom{10}2$, i.e. by \Cref{prop:density_zero}, $\sigma_3(13,121) = 0$.

For the pair $(15, 91)$ we have $91 = \binom{15}3 - \binom{14}{3} =\binom{9}{3} +7$, and for the pair $(15,169 )$ we have $169 = \binom{15}3-\binom{13}3 = \binom{11}3 +4$. Note that for $m= 15$, we have $2m < \binom{9}{2} < \binom{11}{2}$, i.e. by \Cref{prop:density_zero}, $\sigma_3(15,91) = \sigma_3(15, 169) = 0$.
\end{proof}

\Cref{lem:classification} implies that for $r=3$, the smallest value of $m$, for which the first open problem has no answer is $m=6$. In this case, $f=10$ is the only possible value for which we might have 
$\sigma_3(6, f)>0.$ This leads to the following sub-problem of the first open problem: What is $\sigma_3(6,10)$?\\

From \Cref{cor:small_r} we obtain that $\sigma_3(6,10) \le \frac 59$, but we can do slightly better, as can be seen by considering the following construction:

Let $G_1 = K_3^{(3)}$, i.e. the 3-graph on 3 vertices with one edge. Assume $G_{k-1}$ has been constructed.
We obtain $G_k = (V_k, E_k)$ by taking 3 copies of $G_{k-1}$ and adding all edges using exactly one vertex from each copy of $G_{k-1}$, i.e. inserting the edges of a complete $3$-partite graph. Thus, for $G_k$ we have $|V_k| = 3|V_{k-1}| = 3^k$ and $|E_k| = |V_{k-1}|^3 + 3|E_{k-1}|$. 
One can show that $ \frac{|E_k|}{\binom{|V_k|}{3}}  = \frac{1}{4} - o(1),$  see for example B\'ar\'any and F\"uredi~\cite{BF}.
%
%\emph{Claim:} We have $|E_k| =  3^{k-3}\sum_{i=1}^{k} 3^{2i}$.\\
%For $k=1$, we clearly have $|E_1| = 3^{1-3}\sum_{i=1}^1 3^{2i} = 3^{-2}\cdot 3^2 = 1$, so assume the formula holds for some $k \in \mathbb N$. Then we have 
%\begin{align*}|E_{k+1}| &=  |V_k|^{3} + 3|E_{k}| = 3^{3k} + 3\left(3^{k-3}\sum_{i=1}^k 3^{2i}\right) %= 3^{k-2}3^{2(k+1)} + 3^{k-2}\sum_{i=1}^k
%	 3^{2i} = 3^{(k+1)-3}\sum_{i=1}^{k+1} 3^{2i}. \end{align*}
%Thus, we have $|E_k| = 3^{k-3}\sum_{i=1}^{k} 3^{2i} = 3^{k-3}  \frac{9(3^{2k}-1)}{9-1} = \frac 98(3^{3(k-1)} - 3^{k-3})$. In particular, we have 
%$$ \frac{|E_k|}{\binom{|V_k|}{3}} %= \frac{\frac 98(3^{3(k-1)} - 3^{k-3})}{\binom{3^{k}}3} 
%= \frac{3^2}{8}\frac{ 3^{3k-3}}{\frac{3^{3k}}6} - o(1) = \frac{1}{4} - o(1).$$

Thus, for sufficiently large $n = 3^k$, this gives a construction with $(\frac14-o(1))\binom n3$ edges.

Next we show that every $6$-set in $G_n$ spans at most $8$ edges. 
Let $X\subseteq V_n$ be a set of $6$ vertices in $G_n$.

Assume $X$ contains vertices from $3$ distinct copies of $G_{n-1}$. If $X$ contains $2$ vertices from each copy of $G_{n-1}$, then $X$ induces only the transversal edges, i.e. exactly $8$. If $X$
contains $3$, $2$, and $1$ vertices in the different parts, then $X$ induces $6$ transversal edges and at most $1$ additional edge inside one of the $G_{n-1}$'s. If $X$ contains $4$, $1$ and $1$ vertices in the different copies of $G_{n-1}$, then $X$ induces $4$ transversal edges and at most $\binom 43 = 4$ additional edges inside one of the $G_{n-1}$'s. 

If $X$ contains vertices from at most $2$ distinct copies of $G_{n-1}$, then $X$ induces no transversal edges, i.e. all edges induced by $X$ are contained within a copy of $G_{n-1}$.  Thus, we can iterate our argument, and obtain that any set of $6$ vertices induces at most $8$ edges. 

In particular, that shows that any subgraph of $G_n$ also does not induce $(6,10)$. Looking at the complement of $G_n$, we obtain that any $6$ vertices induce at least $\binom 63 - 8 = 12$ edges, and clearly so does any supergraph of the complement of $G_n$. \\
Thus, for $n$ sufficiently large and any $e \le  (\frac14-o(1))\binom n3$ or $e \ge (\frac34-o(1))\binom n3$ there exists a $3$-graph on $n$ vertices with $e$ edges, which does not arrow $(6,10)$. Thus, $\sigma_3(6,10) \le \frac 12$.\\

Note that this construction is not the only one that shows $\sigma_3(6,10) \le \frac12$: Blowing up a $C_5^{(3)}$ in the same manner as $K_3^{(3)}$ above is also induced $(6,10)$-free and achieves the same bound.\\

Very recently it was shown by Axenovich, Balogh, Clemen and the author~\cite{ABCW} that indeed we have $\sigma_3(6,10) > 0$. It would be interesting to further investigate this problem, as currently there is no other known pair $(m,f)$ for $r\ge 3$ which satisfies $\sigma_r(m,f) > 0$.\\

\textbf{Acknowledgements:} The author would like to thank Maria Axenovich for helpful discussions and valuable feedback on the manuscript, as well as Felix Christian Clemen for his comments on the manuscript and his suggestion to extend \Cref{yes-clique-forest}.


\begin{bibdiv} 
	\begin{biblist} 
		\bib{AK}{article}{
			title={Induced subgraphs with distinct sizes},
			author={Alon, Noga},
			author={Kostochka, Alexandr},
			journal={Random Structures \& Algorithms},
			volume={34},
			number={1},
			pages={45--53},
			year={2009},
			publisher={Wiley Online Library}
		}		
		
		
		\bib{ABKS}{article}{
			title={Sizes of induced subgraphs of Ramsey graphs},
			author={Alon, Noga},
			author={Balogh, J{\'o}zsef},
			author={Kostochka, Alexandr},
			author={Samotij, Wojciech},
			journal={Combinatorics, Probability \& Computing},
			volume={18},
			number={4},
			pages={459},
			year={2009}
		}
		
		\bib{AKS}{article}{
			title={Induced subgraphs of prescribed size},
			author={Alon, Noga},
			author={Krivelevich, Michael},
			author={Sudakov, Benny},
			journal={Journal of Graph Theory},
			volume={43},
			number={4},
			pages={239--251},
			year={2003},
			publisher={Wiley Online Library}
		}
		
		\bib{AB}{article}{
			title={Graphs having small number of sizes on induced k-subgraphs},
			author={Axenovich, Maria},
			author={Balogh, J{\'o}zsef},
			journal={SIAM Journal on Discrete Mathematics},
			volume={21},
			number={1},
			pages={264--272},
			year={2007},
			publisher={SIAM}
		}
	
		\bib{ABCW}{article}{
			title={Order-size pairs of positive density in hypergraphs},
			author={Axenovich, Maria},
			author={Balogh, J{\'o}zsef},
			author={Clemen, Felix Christian},
			author={Weber, Lea},
			journal={in preparation},
			year={2022+}
		}
		
		\bib{AW}{article}{
			title={Absolutely avoidable order-size pairs for induced subgraphs},
			author={Axenovich, Maria},
			author={Weber, Lea},
			journal={arXiv preprint arXiv:2106.14908},
			year={2021}
		}
	
	\bib{BF}{article}{
		title={Almost similar configurations},
		author={B{\'a}r{\'a}ny, Imre},
		author={F{\"u}redi, Zolt{\'a}n},
		journal={arXiv preprint arXiv:1805.02072},
		year={2018}
	}

		%probabilistic 2-uniform
		\bib{B}{book}{
		title={Extremal graph theory},
		author={Bollob{\'a}s, B{\'e}la},
		year={2004},
		publisher={Courier Corporation}
	}

\bib{BS}{article}{
	title={Induced subgraphs of Ramsey graphs with many distinct degrees},
	author={Bukh, Boris},
	author={Sudakov, Benny},
	journal={Journal of Combinatorial Theory, Series B},
	volume={97},
	number={4},
	pages={612--619},
	year={2007},
	publisher={Elsevier}
}



\bib{CYZ}{article}{
	title={The feasibility problem for line graphs},
	author={Caro, Yair}, 
	author={Lauri, Josef}, 
	author={Zarb, Christina},
	journal={arXiv preprint arXiv:2107.13806},
	year={2021}
}

\bib{dC}{article}{
	AUTHOR = {de Caen, D.},
	TITLE = {Extension of a theorem of {M}oon and {M}oser on complete
		subgraphs},
	JOURNAL = {Ars Combin.},
	FJOURNAL = {Ars Combinatoria},
	VOLUME = {16},
	YEAR = {1983},
	PAGES = {5--10},
	ISSN = {0381-7032},
	MRCLASS = {05C30},
	MRNUMBER = {734038},
	MRREVIEWER = {E. M. Palmer},
}
		
				%% Erdos-Furedi original paper
		\bib{EFRS}{article}{
			title={Induced subgraphs of given sizes},
			author={Erd{\H{o}}s, Paul},
			author={F{\"u}redi, Zolt{\'a}n},
			author={Rothschild, Bruce}, 
			author={S{\'o}s, Vera},
			journal={Discrete mathematics},
			volume={200},
			number={1-3},
			pages={61--77},
			year={1999},
			publisher={Elsevier}
		}
		
		
		%% He Ma Zhao improvements
		\bib{HMZ}{article}{
			title={Improvements on induced subgraphs of given sizes},
			author={He, Jialin}, 
			author={Ma, Jie}, 
			author={Zhao, Lilu},
			journal={arXiv preprint arXiv:2101.03898},
			year={2021}
		}
	
	\bib{KS1}{article}{
		title={Ramsey graphs induce subgraphs of quadratically many sizes},
		author={Kwan, Matthew},
		author={Sudakov, Benny},
		journal={International Mathematics Research Notices},
		volume={2020},
		number={6},
		pages={1621--1638},
		year={2020},
		publisher={Oxford University Press}
	}
	
	
	\bib{KS2}{article}{
		title={Proof of a conjecture on induced subgraphs of Ramsey graphs},
		author={Kwan, Matthew},
		author={Sudakov, Benny},
		journal={Transactions of the American Mathematical Society},
		volume={372},
		number={8},
		pages={5571--5594},
		year={2019}
	}
	
		%%explicit construction of dense high girth graphs
	\bib{LUW}{article}{
		title={A new series of dense graphs of high girth},
		author={Lazebnik, Felix},
		author={Ustimenko, Vasiliy},
		author={Woldar, Andrew},
		journal={Bulletin of the American mathematical society},
		volume={32},
		number={1},
		pages={73--79},
		year={1995}
	}
	
	\bib{M}{article}{
		title={A hypergraph extension of Tur{\'a}n's theorem},
		author={Mubayi, Dhruv},
		journal={Journal of Combinatorial Theory, Series B},
		volume={96},
		number={1},
		pages={122--134},
		year={2006},
		publisher={Elsevier}
	}

	\bib{MP}{article}{
		title={A new generalization of Mantel's theorem to k-graphs},
		author={Mubayi, Dhruv},
		author={Pikhurko, Oleg},
		journal={Journal of Combinatorial Theory, Series B},
		volume={97},
		number={4},
		pages={669--678},
		year={2007},
		publisher={Elsevier}
	}

\bib{NST}{article}{
	title={Ramsey graphs induce subgraphs of many different sizes},
	author={Narayanan, Bhargav},
	author={Sahasrabudhe, Julian},
	author={Tomon, Istv{\'a}n},
	journal={Combinatorica},
	volume={39},
	number={1},
	pages={215--237},
	year={2019},
	publisher={Springer}
}

\bib{S81}{article}{
	title={Systems of sets that have the T-property},
	author={Sidorenko, A. F.},
	journal={Moscow University Mathematics Bulletin},
	number={36},
	pages={22--26},
	year={1981}
}

%Turan
\bib{T}{article}{
	title={On an extremal problem in graph theory},
	author={Tur{\'a}n, Paul},
	journal={Matematikai \'es Fizikai Lapok (in Hungarian)},
	volume={48},
	pages={436--452},
	year={1941}
}
		

		
	\end{biblist} 
\end{bibdiv}
\end{document}